\documentclass[a4paper,10pt]{article}
\usepackage[latin1]{inputenc}
\usepackage[T1]{fontenc}
\usepackage{amsmath}
\usepackage{amsfonts}
\usepackage{amstext}
\usepackage{amsthm}
\usepackage[all]{xy}
\usepackage{geometry}
\usepackage{srcltx}
\usepackage{graphics}
\usepackage{epsfig}
\usepackage{subfigure}
\usepackage{slashed}
\usepackage{color}
\usepackage{amssymb}
\usepackage{srcltx}
\newtheorem{theorem}{Theorem}[section]
\newtheorem{lemma}[theorem]{Lemma}

\DeclareMathOperator{\ind}{ind}

\DeclareMathOperator{\ch}{ch}

\DeclareMathOperator{\sfl}{sf}

\DeclareMathOperator{\spin}{Spin}
\DeclareMathOperator{\End}{End}

\DeclareMathOperator{\Hom}{Hom}
\DeclareMathOperator{\GL}{GL}
\DeclareMathOperator{\SO}{SO}

\title{On the space of connections having non-trivial twisted harmonic spinors}
\author{Francesco Bei and Nils Waterstraat}

\begin{document}
\date{}
\maketitle

\footnotetext[2]{F. Bei was supported by the SFB 647 ``Space--Time--Matter''.}
\footnotetext[3]{N. Waterstraat was supported by the Berlin Mathematical School and the SFB 647 ``Space--Time--Matter''.}

\begin{abstract}
We consider Dirac operators on odd-dimensional compact spin manifolds which are twisted by a product bundle. We show that the space of connections on the twisting bundle which yield an invertible operator has infinitely many connected components if the untwisted Dirac operator is invertible and the dimension of the twisting bundle is sufficiently large. 
\end{abstract}

\section{Introduction}
Let $(M,g)$ be a compact oriented Riemannian manifold of dimension $n$ endowed with a spin structure. We denote by $\Sigma M$ the associated spinor bundle and by $\slashed D$ the classical Dirac operator acting on sections of $\Sigma M$. The spectrum of $\slashed D$ is well studied (cf.~\cite{Ginoux}), in particular, there exist lower bounds for the smallest eigenvalue of $\slashed D$. For example, Friedrich's inequality states that for $n\geq 2$ any eigenvalue $\lambda$ of $\slashed D$ satisfies

		\begin{align}\label{Friedrich}
			\lambda^2\geq\frac{n}{4(n-1)}S_0,
		\end{align}
where $S_0$ denotes the minimum of the scalar curvature of $M$. Given a Hermitian vector bundle with a compatible connection on $M$, one can define the twisted Dirac operator,	which plays an important role in Mathematical Physics. In their celebrated paper \cite{VaWi}, Vafa and Witten obtained universal upper bounds for the smallest eigenvalues of twisted Dirac operators depending only on	the geometry of the underlying manifold and not on the twisting bundle and its connection. It is a remarkable feature of their work that, although the results are purely analytical, the proofs are quite topological by using the Atiyah-Singer Index Theorem. Moreover, in odd dimensions
	also the spectral flow and its relation with the index come into play, as exposed by Atiyah, Patodi and Singer in \cite{AtiyahPatodi}. A rigorous mathematical treatment of these results was later
	given by Atiyah in \cite{Atiyah} and the ideas were used in other  publications (cf.~\cite{Davaux}, \cite{Herzlich}) concerning similar settings.\\
Lower bounds on the first eigenvalue for twisted Dirac operators were obtained in \cite{Miatello} under strong assumptions on the manifold $(M,g)$ and the connection on the twisting bundle.
	In \cite{Twisted} Jardim and Le\~ao showed that such a lower bound cannot exist in general: it depends strongly on the twisting connection. Indeed, they consider the convex space $\Omega^N$ of all compatible connections on the product bundle $M \times \mathbb{C}^N$ and the continuous functional
	
\begin{align}\label{functional}
	\lambda : \Omega^N \to \mathbb{R}, \qquad \nabla \mapsto |\lambda_1(\nabla)|,
\end{align}

which assigns to each connection the absolute value of its smallest non-vanishing eigenvalue. They assume that $N\geq n$ if $n$ is even, $N\geq\frac{n+1}{2}$ if $n$ is odd, and that the classical (untwisted) Dirac operator is invertible. Based on the methods used by Vafa and Witten in \cite{VaWi}, they show that there exist connections $\nabla_0$ and $\nabla_1$ on $M\times\mathbb{C}^N$ such that the twisted Dirac operator with respect to $\nabla_0$ is invertible while the one with respect to $\nabla_1$ has a non-trivial kernel. Since the functional \eqref{functional} is continuous, this proves that an estimate like \eqref{Friedrich} cannot exist independently of the twisting connection.\\
The aim of this article is to improve the result of Jardim and Le\~ao for odd-dimensional manifolds in the following sense. We assume as in \cite{Twisted} that $\ker\slashed D = \{0\}$ and that the dimension $N$ of the twisting bundle $M \times \mathbb{C}^N$ is not less than half the dimension of the underlying manifold $n$. The space of all compatible connections $\Omega^N$ on the twisting bundle carries a canonical topology when regarded as a subspace of all endomorphism-valued $1$-forms. We show that the space of all elements in $\Omega^N$ such that the associated twisted Dirac operator is invertible has infinitely many connected components.\\ 
The article is structured as follows. In the second section we outline twisted Dirac operators and the index formula, mainly to fix our notation. In the third section we state our theorem and we prove it in the fourth and last section, which is in its turn divided into four subsections. At first, we recall the definition and the main properties of the spectral flow, which is a homotopy invariant for paths of formally self-adjoint elliptic operators. In particular, we recall from \cite{AtiyahPatodi} that the spectral flow for closed paths can be computed as the index of an associated elliptic operator on $M\times S^1$. Subsequently, we discuss briefly some properties of the mapping degree for smooth maps between compact manifolds, and we outline Vafa and Witten's argument from \cite{VaWi} following Atiyah's lecture \cite{Atiyah}. We assemble these facts in order to construct for each integer $k$ a path of connections in $\Omega^N$ such that the corresponding path of twisted Dirac operators has spectral flow $k$. Finally, the result on the number of path components is obtained similarly as in the second-named author's paper \cite{SpinorsIch}.


\section{Twisted Dirac operators and the index formula}\label{Dirac}
Let $(M,g)$ be an oriented manifold of dimension $n\in\mathbb{N}$ and let $\SO(M)$ denote the principal bundle of positively oriented orthonormal frames on $M$. A spin structure on $M$ consists of a principal bundle $\spin(M)$ over $M$ and a map $\spin(M)\rightarrow\SO(M)$, such that the diagram 

\begin{align*}
\xymatrix{
\spin(M)\times\spin_n\ar[r]\ar[dd]&\spin(M)\ar[dd]\ar[dr]&\\
&&M\\
\SO(M)\times\SO_n\ar[r]&\SO(M)\ar[ur]&
}
\end{align*}
commutes, where $\spin_n\rightarrow\SO_n$ denotes the non-trivial two-folded covering of the special orthogonal group and the vertical arrows are given by the corresponding group actions. The $n$-dimensional complex spinor representation $\delta_n:\spin_n\rightarrow\GL(\Sigma_n)$, where $\Sigma_n$ is of dimension $2^{\lfloor\frac{n}{2}\rfloor}$, now defines the spinor bundle of the spin structure as the associated bundle $\Sigma M=\spin M\times_{\Delta_n}\Sigma_n$. The spinor bundle has a natural Clifford multiplication 

\[T_pM\times \Sigma_p\ni(X_p,s_p)\mapsto X_p\cdot s_p\in \Sigma_p,\]
a Hermitian metric $\langle\cdot,\cdot\rangle_\Sigma$ and a metric connection $\nabla^\Sigma$.\\
Let $E$ be a vector bundle over $M$ having a bundle metric $\langle\cdot,\cdot\rangle_E$ and a compatible connection $\nabla^E$. Then the tensor product $\Sigma M\times E$ inherits a Clifford multiplication, a Hermitian metric and a compatible connection. The twisted Dirac operator $D^{(E,\nabla^E)}:\Gamma(\Sigma M\otimes E)\rightarrow\Gamma(\Sigma M\otimes E)$ is defined by the composition of $\nabla^\Sigma\otimes\nabla^E:\Gamma(\Sigma M\otimes E)\rightarrow\Gamma(\Sigma M\otimes E\otimes TM)$ and the operator $\Gamma(\Sigma M\otimes E\otimes TM)\rightarrow\Gamma(\Sigma M\otimes E)$ induced by Clifford multiplication on $\Sigma M\otimes E$. If $\{e^1,\ldots, e^n\}$ is a local orthonormal frame and $\sigma\otimes s\in\Gamma(\Sigma M\otimes E)$, then we have locally

\begin{align}\label{Diraclocal}
\begin{split}
D^{(E,\nabla^E)}(\sigma\otimes s)&=\sum^n_{i=1}{(e^i\cdot\nabla^\Sigma_{e^i}\sigma)\otimes s}+\sum^n_{i=1}{(e^i\cdot\sigma)\otimes\nabla^E_{e^i}s}\\
&=(\slashed D\sigma)\otimes s+\sum^n_{i=1}{(e^i\cdot\sigma)\otimes\nabla^E_{e^i}s},
\end{split}
\end{align}       
where $\slashed D$ denotes the classical, untwisted Dirac operator acting on spinors.\\
The operators $D^{(E,\nabla^E)}$ are formally selfadjoint and consequently their Fredholm indices vanish. However, if $n$ is even, the bundle $\Sigma M\otimes E$ splits as a sum $(\Sigma M_+\otimes E)\oplus(\Sigma M_-\otimes E)$ and the twisted Dirac operators map $\Gamma(\Sigma M_\pm\otimes E)$ into $\Gamma(\Sigma M_\mp\otimes E)$. The restrictions of $D^{(E,\nabla^E)}_\pm$ to $\Gamma(\Sigma M_\pm\otimes E)$ are Fredholm operators and according to the Atiyah-Singer index theorem, their Fredholm indices are

\begin{align}\label{AS}
\ind D^{(E,\nabla^E)}_\pm=\pm(-1)^\frac{n}{2}\int_M{\hat{\mathbf{A}}(M)\,\ch(E)},
\end{align}
where $\ch(E)\in H^{2\ast}(M;\mathbb{Q})$ is the Chern character of the twisting bundle $E$ and $\hat{\mathbf{A}}(M)\in H^{4\ast}(M;\mathbb{Q})$ denotes the total $\hat{\mathbf{A}}$ class of $M$.


\section{The Theorem}
Let $(M,g)$ be an oriented Riemannian manifold of odd dimension $n=2m-1\geq 3$ with a fixed spin structure. We consider for $N\in\mathbb{N}$ the product bundle $\Theta^N=M\times\mathbb{C}^N$ with the Hermitian metric induced by the standard metric on $\mathbb{C}^N$. In what follows, we denote by $d$ the compatible connection on $\Theta^N$ which is given by the ordinary derivative, i.e. 

\[d_Xs=\sum^N_{i=1}{(X.s_i)\,e_i},\quad \text{where}\quad s(p)=\sum^N_{i=1}{s_i(p)\,e_i}\in\mathbb{C}^N,\,p\in M.\]
Then every connection on $\Theta^N$ can be written as $\nabla=d+\alpha$, where $\alpha$ is a smooth section of the bundle $\Hom(TM\otimes\Theta^N,\Theta^N)\approx T^\ast M\otimes\End(\Theta^N)$. Consequently, the space $\Omega^N$ of all compatible connections on $\Theta^N$ is a convex subspace of the space $\Gamma(T^\ast M\otimes\End(\Theta^N))$ with respect to the $C^0$-topology. In what follows we shorten notation by $D^\nabla:=D^{(\Theta^N,\nabla)}$ and we let $\Omega^N_{inv}$ be the set of all $\nabla\in\Omega^N$, such that $D^\nabla$ is invertible, i.e. $\ker D^\nabla$ is trivial. Elements in $\ker D^\nabla$ are called harmonic spinors. The aim of this note is to prove the following theorem.

\begin{theorem}\label{main}
If $\ker \slashed D=\{0\}$ and $N\geq m$, then $\Omega^N_{inv}$ has infinitely many connected components.
\end{theorem}
Let us briefly comment on our assumptions in Theorem \ref{main}. It is a well known result of Lichnerowicz that $\ker\slashed D=\{0\}$ if $g$ is of positive scalar curvature (cf. \cite[Thm. 8.8]{Lawson}), and the classical Dirac operator is also invertible on all flat tori as long as the spin strucure is non-trivial (cf. \cite[Thm. 2.1.1]{Ginoux}). We strongly use that $\ker\slashed D=\{0\}$ in the proof of Theorem \ref{main} and we do not know how to remove this assumption in our argument. Moreover, we also need that $\pi_{2m-1}(U(N))\cong\mathbb{Z}$ if $N\geq m$, and so our arguments do not seem to generalise to lower dimensions of the twisting bundle.


\section{Proof of Theorem \ref{main}}

\subsection{The Spectral flow and its properties}
The spectral flow was defined by Atiyah, Patodi and Singer in \cite{AtiyahPatodi} in connection with the eta-invariant and asymmetry of spectra. Here we introduce it along the lines of \cite{BoWo} (cf. also \cite[\S 17]{BoWoBuch}). Let $W$ and $H$ be Hilbert spaces with a compact dense injection $W\hookrightarrow H$. We denote by $\mathcal{FS}(W,H)$ the subset of the space $\mathcal{L}(W,H)$ of all bounded operators with the norm topology, consisting of those elements which are selfadjoint when regarded as unbounded operators in $H$ with dense domain $W$. In what follows we consider $\mathcal{FS}(W,H)$ as a topological space with respect to the topology induced from $\mathcal{L}(W,H)$ (cf. \cite{Hamiltonian}).\\
Each element in $\mathcal{FS}(W,H)$ has a compact resolvent and hence it is in particular a Fredholm operator. Moreover, its spectrum is a discrete unbounded subset of the real line which consists solely of eigenvalues of finite multiplicity.\\
If $\mathcal{A}:I\rightarrow\mathcal{FS}(W,H)$ is a path, then there exists a countable number of continuous functions $j_k:I\rightarrow\mathbb{R}$, $k\in\mathbb{Z}$, such that $j_k(t)\leq j_{k+1}(t)$ for all $t\in I$ and the spectrum of $\mathcal{A}_t$ is given by $\{j_k(t):\, k\in\mathbb{Z}\}$, where the eigenvalues are listed according to their multiplicities. In what follows we assume that $\mathcal{A}$ has unitarily equivalent endpoints, i.e. there exists a unitary operator $U$ acting on $H$ such that $\mathcal{A}_0=U^\ast\mathcal{A}_1U$. Then $\mathcal{A}_0$ and $\mathcal{A}_1$ have the same eigenvalues with the same multiplicities and consequently there exists $l\in\mathbb{Z}$ such that $j_{k+l}(0)=j_k(1)$ for all $k\in\mathbb{Z}$. This integer is called the \textit{spectral flow} of $\mathcal{A}$ and we denote it by $\sfl(\mathcal{A})$. Let us point out that the spectral flow is the number of negative eigenvalues of $\mathcal{A}_0$ that become positive when the parameter travels from $0$ to $1$, minus the number of positive eigenvalues that become negative.\\
The following basic properties of the spectral flow are fundamental for our proof of Theorem \ref{main}.

\begin{lemma}\label{propSfl}
Let $\mathcal{A},\widetilde{\mathcal{A}}:I\rightarrow\mathcal{FS}(W,H)$ be two paths with unitarily equivalent ends.
\begin{enumerate}
	\item[i)] If $\widetilde{\mathcal{A}}_t=\mathcal{A}_{1-t}$, $t\in I$, then $\sfl(\mathcal{A})=-\sfl(\widetilde{\mathcal{A}})$.
	\item[ii)] If $\mathcal{A}_t=\mathcal{A}_0$ for all $t\in I$, i.e. $\mathcal{A}$ is a constant path, then $\sfl(\mathcal{A})=0$.
	\item[iii)] If $\mathcal{A}_1=\widetilde{\mathcal{A}}_0$, then $\sfl(\mathcal{A}\ast\widetilde{\mathcal{A}})=\sfl(\mathcal{A})+\sfl(\widetilde{\mathcal{A}})$.
	\item[iv)] If $\ker\mathcal{A}_t=\{0\}$ for all $t\in I$, then $\sfl(\mathcal{A})=0$.
	\item[v)] If $\mathcal{A},\widetilde{\mathcal{A}}$ are homotopic through a homotopy leaving the endpoints fixed, then $\sfl(\mathcal{A})=\sfl(\widetilde{\mathcal{A}})$. 
\end{enumerate}
\end{lemma} 

\begin{proof}
At first, the assertions i)-iii) are direct consequences of the definition. In order to show iv), we assume by contradiction that $\sfl(\mathcal{A})\neq 0$. Clearly, under this assumption the spectrum of $\mathcal{A}_0$ is neither a subset of $(-\infty,0)$ nor is it contained in $(0,\infty)$. Since $\sfl(-\mathcal{A})=-\sfl(\mathcal{A})$ by the very definition of the spectral flow, we can suppose that $\sfl\mathcal{A}>0$. Let $j_k:I\rightarrow\mathbb{R}$ be the function in the definition of the spectral flow such that $j_k(0)$ is the largest negative eigenvalue of $\mathcal{A}_0$, and $j_k(1)\geq j_i(1)$ for any $j_i:I\rightarrow\mathbb{R}$ such that $j_i(0)=j_k(0)$ if $j_k(0)$ is not a simple eigenvalue. Since $\sfl\mathcal{A}>0$, we have that $j_k(1)\geq0$. Consequently, there exists $t\in I$ such that $j_k(t)=0$ and hence $\ker\mathcal{A}_t\neq 0$ contradicting our assumption.\\
Finally, we show assertion v). Let $H:I\times I\rightarrow\mathcal{FS}(W,H)$ be a continuous map, such that $H(0,t)=\mathcal{A}_t$, $H(1,t)=\widetilde{\mathcal{A}}_t$ and $H(\lambda,0)$, $H(\lambda,1)$ are constant for $\lambda\in I$. We define a path $h:I\rightarrow\mathcal{FS}(W,H)$ by the concatenation

\begin{align*}
h_t=\begin{cases}
\mathcal{A}_{2t},\quad 0\leq t\leq\frac{1}{2}\\
\widetilde{A}_{2-2t},\quad \frac{1}{2}\leq t\leq 1
\end{cases}.
\end{align*}   
Then $\sfl(h)=\sfl(\mathcal{A})-\sfl(\widetilde{\mathcal{A}})$ by i) and iii), and $h$ can be deformed into the constant path $\mathcal{A}_0$ through an $\mathcal{A}_0$ preserving homotopy $G(\lambda,t)$, i.e. $G(0,t)=h_t$, $G(1,t)=\mathcal{A}_0$ for all $t\in I$, and $G(\lambda,0)=G(\lambda,1)=\mathcal{A}_0$ for all $\lambda\in I$. Now, if $\{j_k(\lambda,t)\}_{k\in\mathbb{Z}}$ denotes the sequence of eigenvalues of $G(\lambda,t)$ ordered as in the definition of the spectral flow, we obtain for $k\in\mathbb{Z}$

\[j_k(0,0)=j_k(1,0)=j_k(1,1)=j_k(0,1)\]
which shows that $\sfl h=0$, and consequently, $\sfl(\mathcal{A})=\sfl(\widetilde{\mathcal{A}})$.
\end{proof}

Let $M$ be a closed Riemannian manifold and $B=\{B_t\}_{t\in I}$ a path of formally selfadjoint elliptic operators acting on sections of a Hermitian bundle $E$ over $M$, such that the coefficients depend smoothly on the parameter $t\in I$. Setting $W:=H^1(M,E)$ and $H:=L^2(M,E)$, we are in the situation described above. It was shown by Atiyah, Patodi and Singer in \cite{AtiyahPatodi} that the spectral flow of $B$ can be computed easily if $B_0=B_1$, i.e, if the family is periodic (cf. also \cite[Theorem 17.13]{BoWoBuch}). Let $\pi:M\times S^1\rightarrow M$ be the canonical projection. If we pullback $E$ to $M\times S^1$ by $\pi$, then we obtain an elliptic operator $\hat{B}$ acting on sections of $\pi^\ast E$ by $-\frac{\partial}{\partial t}+B_t$. Now the spectral flow of the path $B$ is given by

\begin{align}\label{sfl=index} 
\sfl(\{B_t\}_{t\in I})=\ind(\hat{B}).
\end{align}
Actually, it can be shown that this formula extends to non-closed paths as follows (cf. \cite[\S 14.2]{Mrowka}): Let again $B=\{B_t\}_{t\in I}$ be a path of formally selfadjoint elliptic operators as above. Now let us no longer assume that $B_0=B_1$, but that there exists a unitary bundle automorphism $F:E\rightarrow E$ such that $B_0=F^\ast B_1F$ which means in particular that $B$ has unitarily equivalent endpoints. We consider on $I\times E$ the equivalence relation

\[(t,u)\sim(s,v):\Longleftrightarrow t=0, s=1;\quad u=Fv.\]
Then the quotient is a Hermitian vector bundle $\hat{E}$ over $M\times S^1$ and the spectral flow of the path $B$ is given by the index of the differential operator $-\frac{\partial}{\partial t}+B_t$ acting on sections of $\hat{E}$.

\subsection{The mapping degree}
The aim of this section is to recall briefly the mapping degree for maps between oriented compact manifolds $M$ and $N$ of the same dimension $n$. Since $H^n(M)$ and $H^n(N)$ are infinitely cyclic, for every continuous map $f:M\rightarrow N$ there exists a unique integer $k$ such that the induced map $f^\ast:H^n(N)\rightarrow H^n(M)$ is multiplication by $k$. We denote this integer by $\deg(f)$. It follows from the definition that if $P$ is another compact oriented manifold and $g:N\rightarrow P$ continuous, then

\begin{align}\label{degmult}
\deg(g\circ f)=\deg(g)\deg(f),
\end{align} 
and moreover, $|\deg(f)|=1$ if $f$ is a homeomorphism. We will use below that if $M$, $N$ are smooth oriented compact manifolds and $f:M\rightarrow N$ is a smooth map, then

\begin{align}\label{degform}
\int_Mf^\ast\omega=\deg(f)\,\int_N\omega
\end{align} 
for any $n$-form $\omega$ on $N$. Moreover, given four compact smooth and oriented manifolds $M$, $M'$, $N$ and  $N'$ with $\dim(M)=\dim(M')=m$, $\dim(N)=\dim(N')=n$ and two maps $f:M\rightarrow M'$ and $g:N\rightarrow N'$, the degree of the product map $f\times g:M\times N\rightarrow M'\times N'$ satisfies 

\begin{align}\label{degprod}
\deg(f\times g)=\deg(f)\deg(g).
\end{align}
Indeed, let $\omega$ be a generator of $H^m(M')$ and let $\eta$ be a generator of $H^n(N')$. Then $\omega\otimes\eta$ is a generator of $H^{m+n}(M'\times N')$ and we have 

\[\int_{M\times N}(f\times g)^*(\omega\otimes \eta)=\left(\int_Mf^*\omega\right)\left(\int_Ng^*\eta\right)=\deg(f)\deg(g).\]  
Finally, we recall the following theorem of Hopf which we will need in the next section (cf. \cite[V.11.6]{Bredon}).

\begin{theorem}\label{thm-deg}
Let $M$ be a compact oriented smooth manifold of dimension $n\in\mathbb{N}$. Then for every $k\in\mathbb{Z}$ there exists a map $f_k:M\rightarrow S^n$ such that $\deg(f_k)=k$.
\end{theorem}


\subsection{Vafa and Witten's argument revisited}
As before we denote by $d$ the trivial connection in the bundle $\Theta^N$ over $M$, and we note that $\ker D^d=\{0\}$ since $\slashed D$ and $D^d$ have the same eigenvalues up to multiplicities. The following argument follows Vafa and Witten \cite{VaWi}, but we particularly refer to Atiyah \cite{Atiyah} and Jardim and Le\~ao \cite{Twisted}.\\
Since $2m-1$ is by assumption in the stable range of $U(N)$, there exists a generator $\tilde{F}:S^{2m-1}\rightarrow U(N)$ of $\pi_{2m-1}(U(N))\cong\mathbb{Z}$. Moreover, by Theorem \ref{thm-deg} there is for every $k\in\mathbb{Z}$ a map $f_k:M\rightarrow S^{2m-1}$ of degree $k$, and we denote by $F_k=\tilde{F}\circ f_k:M\rightarrow U(N)$ the composition. Note that $F_k$ induces a unitary isomorphism $L^2(M,\Theta^N)\rightarrow L^2(M,\Theta^N)$ which we denote also by $F_k$ in what follows.\\
We now consider the path of connections $\{\nabla^k_t\}_{t\in I}$ defined by 

\begin{align}\label{connections}
\nabla^k_t:=d+t\cdot F^{-1}_kdF_k\in\Omega^N
\end{align}
and we set for simplicity of notation $D^k_t:=D^{\nabla^k_t}$. By an elementary computation, we obtain from \eqref{Diraclocal} that

\[D^k_t=D^d+t\cdot X,\]
where $X$ is a multiplication operator which is locally given by 

\[X(\sigma\otimes s)=\sum^{2m-1}_{i=1}{(e_i\cdot\sigma)\otimes (F^{-1}_k(\partial_iF_k)s)}.\]
In particular, the path $\{D^k_t\}_{t\in I}$ is continuous, and moreover, it is readily seen that $D^k_1=F^\ast_k D^d F_k$ so that $D^k_1$ and  $D^k_0=D^d$ are unitarily equivalent. Consequently, the spectral flow of $\{D^k_t\}_{t\in I}$ is defined and we can apply \eqref{sfl=index}: Consider on the product $\mathbb{C}^N\times S^{2m-1}\times I$ the equivalence relation

\[(u,p,t)\sim(v,q,s):\Longleftrightarrow p=q,\, v=\tilde F(p)u,\, t=0,s=1.\]
The quotient is a bundle $H$ over $S^{2m-1}\times S^1$ and we denote its pullback to $M\times S^1$ by $f_k\times id$ by $E^k$. Now the spectral flow of $D^k$ is the index of the operator $-\frac{\partial}{\partial t}+D^k_t$ acting on sections of $E^k$.\\
By definition, the connections $\nabla^k_t$ descend to a connection $\nabla^k$ on $E^k$ and hence we obtain a twisted Dirac operator $D^{(E^k,\nabla^k)}$ on $M\times S^1$ such that the restriction of $D^{(E^k,\nabla^k)}$ to $M\times\{t\}$ is $D^k_t$. We infer that the spectral flow of $D^k$ is given by the index of $D^{(E^k,\nabla^k)}_+$, which can be obtained from \eqref{AS}. Clearly, $c_i(H)\in H^{2i}(S^{2m-1}\times S^1;\mathbb{Z})$ can only be non-trivial if $i=m$, and moreover, Vafa and Witten's argument in \cite{VaWi} uses the fact that

\[\int_{S^{2m-1}\times S^1}{c_{m}(H)}=\ind D^{(E^1,\nabla^1)}=1.\]
Consequently, since $\int_{M\times S^1}\hat{\mathbf{A}}(M\times S^1)=0$ ((cf. \cite[\S 11]{Lawson})), we obtain from 

\[c_i(E^k)=(f_k\times id)^\ast c_i(H),\]
\eqref{degform} and \eqref{degprod}

\begin{align*}
\ind D^{(E^k,\nabla^k)}&=\int_{M\times S^1}\hat{\mathbf{A}}(M\times S^1)\,\ch E^k=\deg(f_k\times id)\int_{S^{2m-1}\times S^1}{c_{m}(H)}=\deg f_k=k.
\end{align*}


\subsection{Proof of Theorem \ref{main}} 
We have defined in \eqref{connections} for each $k\in\mathbb{Z}$ a path of connections $\{\nabla^k_t\}_{t\in I}$ on $\Theta^N$ such that the spectral flow of the associated path of twisted Dirac operators $\{D^k_t\}_{t\in I}$ is $k$. Moreover, the connections at the endpoints $\nabla^k_0=d$ and $\nabla^k_1$ belong to $\Omega^N_{inv}$, since $D^k_1$ is unitarily equivalent to the invertible Dirac operator $D^k_0=D^d$.\\
We now claim that the endpoints $\nabla^k_1$ and $\nabla^l_1$ belong to different path components of $\Omega^N_{inv}$ if $k\neq l$. For, otherwise there exists a path $\{\nabla_t\}_{t\in I}$ in $\Omega^N_{inv}$ such that $D^{\nabla_0}=D^{\nabla^k_1}$ and $D^{\nabla_1}=D^{\nabla^l_1}$. Since $\ker D^{\nabla_t}=0$ for all $t\in I$, we infer from Lemma \ref{propSfl} iv) that 

\begin{align}\label{sfl0}
\sfl(\{D^{\nabla_t}\}_{t\in I})=0.
\end{align}
Now we consider the closed path $\{\widetilde{\nabla}_t\}_{t\in I}$ in $\Omega^N$ defined by 

\begin{align*}
\widetilde{\nabla}_t=\begin{cases}
\nabla^k_{3t},\quad 0\leq t\leq \frac{1}{3}\\
\nabla_{3t-1},\quad \frac{1}{3}\leq t\leq \frac{2}{3}\\
\nabla^l_{-3t+3},\quad \frac{2}{3}\leq t\leq 1
\end{cases}.
\end{align*}
Since $\widetilde{\nabla}$ is homotopic to a constant path in the convex space $\Omega^N$, we deduce from Lemma \ref{propSfl} v) and ii) that $\sfl(\{D^{\widetilde{\nabla}_t}\}_{t\in I})=0$. This implies by iii) and i) of the same lemma that

\[0=\sfl(\{D^{\widetilde{\nabla}_t}\}_{t\in I})=\sfl(\{D^{\nabla^k_t}\}_{t\in I})+\sfl(\{D^{\nabla_t}\}_{t\in I})-\sfl(\{D^{\nabla^l_t}\}_{t\in I})\]
Finally, it follows from \eqref{sfl0} that

\[k=\sfl(\{D^{\nabla^k_t}\}_{t\in I})=\sfl(\{D^{\nabla^l_t}\}_{t\in I})=l\]
which is a contradiction.

\vspace{1cm}
Francesco Bei\\
Institut f\"ur Mathematik\\
Humboldt Universit\"at zu Berlin\\
Unter den Linden 6\\
10099 Berlin\\
Germany\\
E-mail: bei@math.hu-berlin.de

\vspace{1cm}
Nils Waterstraat\\
Institut f\"ur Mathematik\\
Humboldt Universit\"at zu Berlin\\
Unter den Linden 6\\
10099 Berlin\\
Germany\\
E-mail: waterstn@math.hu-berlin.de

\end{document}